\documentclass[12pt]{article}

\usepackage{amsmath,amssymb}
\usepackage{color}

\newtheorem{theorem}{Theorem}[section]
\newtheorem{corollary}[theorem]{Corollary}

\newtheorem{example}{Example}[section]

\begin{document}

\title{\textbf{Steiner (revised) Szeged index of graphs}}

\author{
Modjtaba Ghorbani$^{a}$, Xueliang Li$^{b, c}$, Hamid Reza Maimani$^{a}$,\\Yaping
Mao$^{c}$,  Shaghayegh Rahmani$^{a}$, Mina Rajabi-Parsa$^{a}$
 }

\date{\today}

\maketitle

\begin{center}
$^a$Department of Mathematics, Faculty of Science \\
Shahid Rajaee Teacher Training University\\
Tehran, 16785-136, I.R. Iran\\
\medskip
$^b$Center for Combinatorics and LPMC\\
Nankai University, Tianjin 300071, China\\
\medskip
$^c$School of Mathematics and Statistics\\
Qinghai Normal University\\
Xining, Qinghai 810008, China
\medskip
\end{center}

\begin{abstract}
The Steiner distance in a graph, introduced by Chartrand et al. in
1989, is a natural generalization of the concept of classical graph
distance. For a connected graph $G$ of order at least 2 and
$S\subseteq V(G)$, the Steiner distance $d_G(S)$ of the set $S$ of vertices
in $G$ is the minimum size of a connected subgraph whose vertex set contains
or connects $S$.
In this paper, we introduce the concept of the Steiner (revised) Szeged
index ($rSz_k(G)$) $Sz_k(G)$ of a graph $G$, which is a natural generalization of the well-known (revised) Szeged index of chemical use. We determine the
$Sz_k(G)$ for trees in general. Then we give a formula
for computing the Steiner Szeged index of a graph in terms of orbits of
automorphism group action on the edge set of the graph. Finally, we give sharp upper and lower bounds of
($rSz_k(G)$) $Sz_k(G)$ of a connected graph $G$, and establish some of its properties.
Formulas of ($rSz_k(G)$) $Sz_k(G)$ for small and large $k$ are also
given in this paper.
\end{abstract}

\noindent {\bf Keywords:} distance, Steiner distance, (revised) Szeged index, Steiner (revised) Szeged index.

\medskip\noindent
{\bf AMS Subject Classification 2010:} 05C05, 05C12, 05C35.

\section{Introduction}\label{sec:intro}

All graphs in this paper are assumed to be undirected, finite and
simple. We refer to \cite{Bondy} for graph theoretical notation and
terminology not specified here. Distance is one of basic concepts in
graph theory \cite{BuHa}. If $G$ is a connected graph and $u,v\in
V(G)$, then the \emph{distance} $d(u,v)=d_G(u,v)$ between $u$ and
$v$ in $G$ is the length of a shortest path of $G$ connecting $u$ and $v$. For
more details on classical distance, see \cite{Goddard}.

Let $e= uv$ be an edge of a graph $G$. Let $N_u(e|G)$ be the set of
vertices of $G$ which are closer to $u$ than to $v$ and let
$N_v(e|G)$ be set of those vertices which are closer to $v$ than to $u$.
The set of those vertices which have equal distance from $v$ and $u$ is denoted
by $N_0(e|G)$. More formally,
$$
N_u(e|G)=\{w\in V(G):~d_G(w,u)<d_G(w,v)\},
$$
$$
N_v(e|G)=\{w\in V(G):~d_G(w,v)<d_G(w,u)\}
$$
and
$$
N_0(e|G)=\{w\in V(G):~d_G(w,u)=d_G(w,v)\}.
$$
Let $n_u(e)= |N_u(e|G)|$, $n_v(e)= |N_v(e|G)|$ and $n_0(e)= |N_0(e|G)|$. Then the \emph{Szeged index} of a graph
$G$, denoted by $Sz(G)$, is defined as
$$
Sz(G)=\sum_{e=uv\in E(G)}n_u(e)n_v(e),
$$
and the \emph{revised Szeged index} of a graph $G$, denoted by $rSz(G)$, is defined as
$$
rSz(G)=\sum_{e=uv\in E(G)}(n_u(e)+n_0(e)/2)(n_v(e)+n_0(e)/2).
$$
The basic properties of the (revised) Szeged index and bibliography on ($rSz(G)$) $Sz(G)$ are presented in
\cite{AH, arX, GutmanDobrynin, PR, XZ}.

The Steiner distance of a graph, introduced by Chartrand et al. in
\cite{Chartrand} in 1989, is a natural and nice generalization of
the concept of the classical graph distance. For a graph $G=(V,E)$
and a set $S\subseteq V$ of at least two vertices, \emph{an
$S$-Steiner tree} or \emph{a Steiner tree connecting $S$} (or
simply, \emph{an $S$-tree}) is a subgraph $T=(V',E')$ of $G$ that is
a tree with $S\subseteq V'$. Let $G$ be a connected graph of order
at least $2$ and let $S$ be a nonempty set of vertices of $G$. Then
the \emph{Steiner distance} $d_G(S)$ in $G$ among the vertices of $S$ (or
simply the distance of $S$) is the minimum size of a connected
subgraph whose vertex set contains or connects $S$. Note that if $H$ is a
connected subgraph of $G$ such that $S\subseteq V(H)$ and
$|E(H)|=d(S)$, then $H$ is a tree. Clearly,
$d_G(S)=\min\{|E(T)|\,\,,\,S\subseteq V(T)\}$, where $T$ is a
subtree of $G$. Furthermore, if $S=\{u,v\}$, then $d_G(S)=d(u,v)$ is
nothing new, but the classical distance between $u$ and $v$ in $G$.
Clearly, if $|S|=k$, then $d_G(S)\geq k-1$. For more details on the
Steiner distance, we refer to \cite{Ali,Caceresa, Chartrand, Goddard, Oellermann}.

In \cite{LMG}, Li et al. proposed a generalization of the concept of Wiener
index, using Steiner distance. Thus, the \emph{$k$th
Steiner Wiener index} $SW_k(G)$ of a connected graph $G$ is defined
by
$$
SW_k(G)=\sum_{\overset{S\subseteq V(G)}{|S|=k}} d(S)\,.
$$
For $k=2$, the Steiner Wiener index coincides with the ordinary
Wiener index. It is usual to consider $SW_k$ for $2 \leq k \leq
n-1$, but the above definition implies $SW_1(G)=0$ and $SW_n(G)=n-1$
for a connected graph $G$ of order $n$. For more details on the Steiner
Wiener index, we refer to \cite{LMG, LMG2, MWG, MWGL}.

Let $G$ be a connected graph and $e$ an edge of $G$. For a positive integer $k$,
from the Steiner distance, we define another three sets $N_u(e;k)$, $N_v(e;k)$ and $N_0(e;k)$
as follows.
$$
N_u(e;k) = \{S'\subseteq V(G), |S'|=k-1\,|\,d_G(S'\cup
\{u\})<d_G(S'\cup \{v\}), \ u\notin S', \ v\notin S'\},
$$
$$
N_v(e;k) = \{S'\subseteq V(G), |S'|=k-1\,|\,d_G(S'\cup
\{v\})<d_G(S'\cup \{u\}), \ v\notin S', \ u\notin S'\},
$$
and
$$
N_0(e;k) = \{S'\subseteq V(G), |S'|=k-1\,|\,d_G(S'\cup
\{u\})=d_G(S'\cup \{v\}), \ u\notin S', \ v\notin S'\}.
$$

Let $n_u(e;k)=|N_u(e;k)|$, $n_v(e;k)=|N_v(e;k)|$ and  $n_0(e;k)=|N_0(e;k)|$.
Then the \emph{$k$th Steiner Szeged index} of a graph $G$ is defined as
$$
Sz_{k}(G) = \sum_{e=uv\in E(G)} (n_u(e;k)+1) (n_v(e;k)+1).
$$
Analogously, the \emph{$k$th Steiner revised Szeged index} of a graph $G$ is defined as
$$
rSz_k(G)=\displaystyle\sum_{e=uv\in E(G)}(n_u(e;k)+n_0(e;k)/2+1)(n_v(e;k)+n_0(e;k)/2+1).
$$

Here, one may note that the formula is not the same as the classical
Szeged index in form. If $k=2$, then
$$
N_u(e;2) = \{w\in V(G)\,|\,d_G(u,w)<d_G(v,w), \ u\neq w, v\neq w\}.
$$
One can see $N_u(e;2)\neq N_u$ since we require $u\neq w$. By our
definition, the classical Szeged index $Sz(G)$ can be written as
$$
Sz(G)=Sz_2(G)=\sum_{e=uv\in E(G)}(n_u(e;2)+1)(n_v(e;2)+1),
$$
where $N_u(e;2) = \{w\in V(G)\,|\,d_G(u,w)<d_G(v,w), \ u\neq w\}$ and
$N_v(e;2) = \{w\in V(G)\,|\,d_G(v,w)<d_G(u,w), \ u\neq w\}$.

So, as one can easily see that the Steiner (revised) Szeged index is a natural generalization of the well-known (revised) Szeged index of chemical use.

We proceed as follows. In the next section, we determine the
$Sz_k(G)$ for trees in general. Then, we give a formula
for computing the Steiner Szeged index of a graph in terms of orbits of
automorphism group action on the edge set of the graph. Finally, we give sharp upper and lower bounds of
($rSz_k(G)$) $Sz_k(G)$ of a connected graph $G$, and establish some of its properties.
Formulas of ($rSz_k(G)$) $Sz_k(G)$ for small and large $k$ are also
given.

\section{Results for trees}

At first, we consider trees. The following result is easy to obtain.
\begin{theorem}\label{2}
For a tree $T$,
$$
Sz_{k}(T)=\sum_{e=uv\in E(T)} \left(\displaystyle{n_u(e)-1 \choose
k-1}+1\right)  \left(\displaystyle{n_v(e)-1\choose k-1}+ 1\right),
$$
where $2\leq k\leq |V(T)|-1$.
\end{theorem}

Note that for $k=2$, $Sz_2(T)=\sum_{e=uv\in E(T)} n_u(e) n_v(e)=Sz(T),$ which is exactly the classical Szeged index.

\begin{proof}
Let $T_u$ and $T_v$ be the two components of $T-e$. For any
$(k-1)$-subset $S$ of $V(T)\setminus \{u,v\}$, if both $S\cap
T_u\neq \emptyset$ and  $S\cap T_v\neq \emptyset$, then $d_T(S\cup
u)=d_T(S\cup v)$. So, $d_T(S\cup u)<d_T(S\cup v)$ if and only if $S$
is in $T_u-u$, and $d_T(S\cup v)<d_T(S\cup u)$ if and only if $S$ is
in $T_v-v$. Since $T_u-u$ and $T_v-v$ have $n_u(e)-1$ and $n_v(e)-1$
vertices, respectively, we are thus done.
\end{proof}

Some examples are given as follows.
\begin{example}
For a path $P_n=u_1u_2\cdots u_iu_{i+1}\cdots u_n$ on vertices, take an edge $e=u_iu_{i+1}$. Then $P_n-e$ has two subpaths
$P_i$ and $P_{n-i}$. So we have $n_{u_i}(e)=i$ and $n_{u_{i+1}}(e)=n-i$. Therefore,
$$
Sz_{k}(P_n)=\sum_{i=1}^{n-1} \left(\displaystyle{i-1 \choose
k-1}+1\right) \left(\displaystyle{n-i-1\choose k-1}+ 1\right).
$$
Since any $(k-1)$-subset $S$ of $V(T)\setminus \{u,v\}$ satisfies
$d_T(S\cup u)=d_T(S\cup v)$ if and only if both $S\cap T_u\neq
\emptyset$ and  $S\cap T_v\neq \emptyset,$  then we can deduce that
$$
n_0(e;k)=\sum_{j=1}^{k-2} \displaystyle{i-1 \choose j} \displaystyle{n-i-1\choose k-j-1}.
$$
From this one can give an explicit formula for the $rSz_{k}(P_n)$.
\end{example}

\begin{example}
For the star graph $S_{n+1}$ on $n+1$ vertices with a central vertex $u$ and the other pendant vertices $u_1, u_2, \cdots, u_n$,
take an edge $e=uu_i$. Then $S_{n+1}-e$ has two subgraphs
$T_{u_i}=P_1$ and $T_u=S_n$. So we have $n_{u_i}(e)=1$ and $n_u(e)=n$. Therefore,
$$Sz_{k}(S_{n+1})=\sum_{i=1}^{n} \left(\displaystyle{n-1\choose k-1}+
1\right)= n \displaystyle{n-1\choose k-1} +n.$$
Since any $(k-1)$-subset $S$ of $V(S_{n+1})\setminus \{u,u_i\}$
satisfies $d_T(S\cup u)=d_T(S\cup u_i)$ if and only if both $S\cap
T_u\neq \emptyset$ and $S\cap T_{u_i}\neq \emptyset,$ then we have
$n_0(e;k)=0$ for any $e=uu_i$ because $T_{u_i}-u_i=\emptyset$, and
hence there is no such an $S$. Therefore, we have
$$rSz_k(S_{n+1})=Sz_k(S_{n+1}) = n \displaystyle{n-1\choose k-1} +n.$$
This will be re-obtained next section by using symmetry on graphs.
\end{example}

\noindent {\bf Conjecture:} For any two trees $T$ and $T'$,
$Sz_k(T)\leq Sz_k(T')$ if and only if $Sz(T)\leq Sz(T')$ ?

\section{Results for graphs with symmetry}

Let $G$ be a group and $\Omega$ be a non-empty set. An action of $G$
on $\Omega$, denoted by $(G|\Omega)$, induces a group homomorphism
$\varphi$ from $G$ into the symmetric group $S_{\Omega }$ on
$\Omega$, where $\varphi(g)^{\alpha }=g^{\alpha},~ (\alpha \in
\Omega)$. The \emph{orbit} of an element $\alpha \in \Omega $ is
denoted by $\alpha^{G}$ and it is defined as the set of all
$\alpha^{g}, g\in G$.

A bijection $\sigma$ on the vertex set of a graph $\Gamma$ is named an
\emph{graph automorphism} if it preserves the edge set of $\Gamma$. In other
words, $\sigma$ is a graph automorphism of $\Gamma$ if $e=uv$ is an edge of $\Gamma$ if and only if
$\sigma(e)=\sigma(u)\sigma(v)$ is an edge of $\Gamma$. Let
$Aut(\Gamma)$ be the set of all graph automorphisms of $\Gamma$. Then $Aut(\Gamma)$ under
the composition of mappings forms a group. A graph $\Gamma$ is
called \emph{vertex-transitive} if $Aut(\Gamma)$ acting on $V(G)$ has one orbit. We can
similarly define an edge-transitive graph just by considering $Aut(\Gamma)$ acting on $E(G)$.

By a minimal tree for a sequence of vertices $(v_{1},\cdots, v_{n})$, we mean a
tree containing the vertices $(v_{1},\cdots, v_{n})$ which has the
minimum number of edges.

\begin{theorem}\label{th2-1}
Let $E_{1},\cdots, E_{r}$ be the orbits of a graph $\Gamma$ under the
action of $Aut(\Gamma)$ on the edge set $E(\Gamma)$ of $\Gamma$. Suppose $e=uv$ and
$f=xy$ are two arbitrary edges of $E_{i}$ $(1\leq i \leq r)$. Then
$\lbrace n_{u},n_{v} \rbrace = \lbrace n_{x},n_{y} \rbrace $.
\end{theorem}
\begin{proof}
Since $e$ and $f$ are in the same obit, there is an automorphism $\varphi\in Aut(\Gamma)$ such that
$\varphi(u)=x$ and $\varphi(v)=y$. For every minimal tree $T$
containing the vertices $(u,u_{1}, \cdots, u_{k-1})$, $\varphi(T)$
is a tree that contains $(x,\varphi(u_{1}), \cdots,
\varphi(u_{k-1}))$. This means that $\varphi(T)\in
\mathcal{S}_{k}^{\varphi(u)}$, and thus $n_{u}= \vert
\mathcal{S}_{k}^{u}\vert = \vert \mathcal{S}_{k}^{x}\vert = n_{x}$.
By a similar argument, one can see that $n_{v}= \vert
\mathcal{S}_{k}^{v}\vert = \vert \mathcal{S}_{k}^{y}\vert = n_{y}$.
This means that $\lbrace n_{u},n_{v} \rbrace = \lbrace n_{x},n_{y} \rbrace.$
\end{proof}

The following corollary is immediate.
\begin{corollary}\label{cor2-2}
Let $E_{1},\cdots, E_{r}$ be the orbits of a graph $\Gamma$ under the
action of $Aut(\Gamma)$ on the edge set $E(\Gamma)$ of $\Gamma$ and
$u_{i}v_{i}=e_{i}\in E_{i}$. Then
$$Sz_{k}(\Gamma) = \sum_{i=1}^{r} \vert E_{i} \vert (n_{u_i}(e_i;k)+1)(n_{v_{i}}(e_i;k)+1),$$
and
$$rSz_{k}(\Gamma) = \sum_{i=1}^{r} \vert E_{i} \vert (n_{u_i}(e_i;k)+n_0(e_i;k)/2+1)(n_{v_i}(e_i;k)+n_0(e_i;k)/2+1).$$
\end{corollary}


\begin{example}
Suppose $K_{n}$ is the complete graph on $n$ vertices. It is not
difficult to see that for any $uv=e \in E(K_{n})$, we have
$n_u(e;k)=n_v(e;k)=0$ and $n_0(e;k)=\displaystyle{n-2\choose k-1}$. Then
$$Sz_{k}(K_{n})=\sum_{e=uv\in E(K_{n})} (n_u(e;k)+1)(n_v(e:k)+1) = \vert E(K_{n}) \vert = n(n-1)/2,$$
and
\begin{eqnarray*}
rSz_{k}(K_{n})&=&\sum_{e=uv\in E(K_{n})} (n_u(e;k)+n_0(e;k)/2+1)(n_v(e;k)+n_0(e;k)/2+1)\\
&=& \vert E(K_{n}) \vert\displaystyle{n-2\choose k-1}^{2}.
\end{eqnarray*}
\end{example}
\begin{example}
Suppose $K_{1,n}$ is the star graph on $n+1$ vertices. Let $V(K_{1,n})= \lbrace u, u_{1}, \cdots, u_{n} \rbrace $
and $E(K_{1,n})= \lbrace \lbrace u,u_{1} \rbrace, \cdots, \lbrace
u,u_{n} \rbrace \rbrace $. Again $K_{1,n}$ is edge-transitive and
for any edge $uu_{i}=e_{i} \in E(K_{1,n})$, we have $n_u(e;k)=\displaystyle{n-1\choose k-1}$, $n_{u_i}(e;k)=0$ and $n_0(e;k)=0$. Then
$$rSz_{k}(K_{1,n})=Sz_{k}(K_{1,n})=\sum_{e\in E(K_{n})} \left(\displaystyle{n-1\choose k-1} +1\right) =  n \displaystyle{n-1\choose k-1} +n.$$
See Example 2.2, we get the same result.
\end{example}

For complete multipartite graphs, we can get the exact value for the
$k$th Steiner Szeged index.
\begin{theorem}\label{1}
Let $\Gamma=K_{a_1, a_2,\ldots, a_m}$ be a complete multipartite
graph and let $k$ be an integer such that $k\leq a_i ~ (1 \leq i
\leq m)$. Then
\begin{equation*}
 Sz_k(\Gamma)=\displaystyle\sum_{i=1}^{m-1}\displaystyle\sum_{j=i+1}^{m}a_ia_j\left(\binom{a_i-1} {k-1}+1\right)
\left(\binom{a_j-1} {k-1}+1\right),
\end{equation*}
and
\begin{equation*}
rSz_k(\Gamma)=\displaystyle\sum_{i=1}^{m-1}\displaystyle\sum_{j=i+1}^{m}a_ia_j\left(\binom{a_i-1} {k-1}+n_0(e;k)/2+1\right)
\left(\binom{a_j-1} {k-1}+n_0(e;k)/2+1\right),
\end{equation*}
where $B=V(\Gamma)-(A_i\cup A_j)$ and
\begin{equation*}
 n_0(e;k)=\binom{|B|} {k-1}+\displaystyle\sum_{p=1}^{a_i-2} \displaystyle\sum_{q=1}^{k-1-p} \binom{a_i-2} {p}\binom{a_j-1} {q}
\binom{|B|} {k-1-(p+q)}.
\end{equation*}

\end{theorem}

\begin{proof}
For $\Gamma=K_{a_1, a_2,\ldots , a_m}$, let $A_t \ (1\leq  t \leq m)$
be the multi-partition of $\Gamma$ such that $A_t=\{ a_{t1}, a_{t2}, \ldots ,
a_{ta_t}\}$. Consider two different parts $A_i$ and $A_j$, where
$1\leq i, j \leq m$ and $a_i \leq a_j$. First, let $k \leq a_i$ and
consider the edge $e=uv$ such that $u\in A_i$ and $v\in A_j$.
Suppose that $W \subseteq V(\Gamma)$, where $|W|=k-1$. Let
$W\subseteq A_i$ such that $u \notin W$ and without loss of
generality, we can suppose $W=\{ a_{i1}, a_{i2}, \ldots,
a_{i(k-1)}\}$. Then the tree induced by the edges $ \{ v a_{i1},v
a_{i2},\ldots , v a_{i(k-1)} \}$ is the Steiner tree containing $u$
and the tree induced by the edges $\{ v a_{i1},v a_{i2},\ldots, v
a_{i(k-1)},vu \}$ is the Steiner tree containing $v$. So, $d_S(v) <
d_S(u)$. Similarly, if $W\subseteq A_j$, then $d_S(u) < d_S(v)$. Let
$W=\{w_1, w_2,\ldots, w_{k-1}\} \subseteq V(\Gamma)$ such that
$W\cap (A_i\cup A_j)=\phi$. So, the tree induced by the edges $\{
uw_1, uw_2,\ldots, uw_{(k-1)}\}$ is the Steiner tree containing $u$
and the tree induced by the edges $\{ vw_1, vw_2,\ldots, vw_{(k-1)}\}$
is the Steiner tree containing $v$. This means that $d_S(v) = d_S(u)$.
Also, if $|W\cap A_i |=p$, $| W\cap A_j|=q$ and $|W \cap \left(V(\Gamma)-(A_i\cup A_j)\right)|=l$, where
$W= \{ a_{i1}, a_{i2}, \ldots, a_{ip}, a_{j1}, a_{j2}, \ldots, a_{jq}, w_1, w_2, \ldots, w_l\}$ ($p+q+l=k-1$),
then $\{ua_{j1}, \ldots, ua_{jq}, uw_1, \ldots, uw_l, w_1a_{i1}, \ldots, w_1a_{ip}\}$
is the Steiner trees containing $u$ and
$\{va_{i1}, \ldots, va_{ip}, vw_1, \ldots, vw_l, w_1a_{j1}, \ldots, w_1a_{jq}\}$
is the Steiner trees containing $v$. This implies that $d_S(v) = d_S(u)$.
By the above discussion we have that $n_v(e;k)=\binom{a_i-1} {k-1}$ and $n_u(e;k)=\binom{a_j-1} {k-1}$.
So,
\begin{equation*}
Sz_k(\Gamma)=\displaystyle\sum_{i=1}^{m-1}\displaystyle\sum_{j=i+1}^{m}a_ia_j\left(\binom{a_i-1} {k-1}+1\right)
\left(\binom{a_j-1} {k-1}+1\right).
\end{equation*}

Assume that $B=V(\Gamma)-(A_i\cup A_j)$. Then $n_0(e;k)=X+Y$, where $$X=
\binom{|B|} {k-1} ~and ~ Y=\displaystyle\sum_{p=1}^{a_i-2}
\displaystyle\sum_{q=1}^{k-1-p} \binom{a_i-2} {p}\binom{a_j-1} {q}
\binom{|B|} {k-1-(p+q)}.$$
This completes the proof.
\end{proof}

\section{Formulas for large $k$}

For trees, we have the following formula for $k=n-1$.
\begin{theorem}\label{pen}
Let $T$ be a tree of order $n$ with $p$ pendent edges. Then
\begin{equation*}
Sz_{n-1}(T)=n+p-1
\end{equation*}
and
\begin{equation*}
rSz_{n-1}(T)=2p+\frac{9}{4}(n-p-1).
\end{equation*}
\end{theorem}
\begin{proof}
Let $e=uv$ be an edge of $T$. If $e$ is not a leaf, then
$|N_u(e;n-1)|=|N_v(e;n-1)|=0$. Suppose $e$ is a leaf and $u$ is a
pendent vertex. Then $v$ is a cut vertex. Then $|N_v(e;n-1)|=1$ and
$|N_u(e;n-1)|=0$, and hence
$$
Sz_{n-1}(T)=(n-1-p)+2p=n+p-1
$$
and
$$
rSz_{n-1}(T)=2p+\frac{9}{4}(n-p-1).
$$
\end{proof}

\noindent \textbf{Remark 1.} Notice that the derivative function
$(rSz_{n-1})^{'}(T)$ is less than zero thus the function
$rSz_{n-1}(T)=2p+\frac{9}{4}(n-p-1)$ is strictly incrasing. Let
$\mathcal{T}_n$ be all of trees with $n$ vertices. Among all
elements of $\mathcal{T}_n$,  the star graph $S_n$ and the path
graph $P_n$ has the minimum and the maximum velue of $rSz_{n-1}$,
respectively.\\

The following observation is immediate for $k=n-1$.
\begin{theorem}\label{pen}
Let $G$ be a connected graph of order $n$ and size $m$ with $p$
pendent edges. Then
$$
Sz_{n-1}(G)=p+m.
$$
and
\begin{equation*}
rSz_{n-1}(G)=2p+\frac{9}{4}(m-p).
\end{equation*}
\end{theorem}
\begin{proof}
Let $e=uv$ be an edge of $T$. If $e$ is not a pendent edge, then
$|N_u(e;n-1)|=|N_v(e;n-1)|=0$. Suppose $e$ is a pendent edge and $u$
is a pendent vertex. Then $v$ is a cut vertex. Then $|N_v(e;n-1)|=1$
and $|N_u(e;n-1)|=0$, and hence
$$
Sz_{n-1}(T)=(m-p)+2p=p+m.
$$
and
$$
rSz_{n-1}(G)=2p+\frac{9}{4}(m-p).
$$
\end{proof}

\section{Upper and lower bounds}

For general graphs, we have the following upper and lower bounds.
\begin{theorem}\label{pen}
Let $n,k$ be two integers with $2\leq k\leq n-1$, and let $G$ be a
graph of order $n$ and size $m$.

$(1)$ If $G$ is $(n-k)$-connected, then
$$
Sz_{k}(G)=m.
$$

$(2)$ If $G$ is not $(n-k)$-connected, then
$$
m\leq Sz_{k}(G)\leq m\left(\left\lceil\frac{1}{2}{n-2\choose
k-1}\right\rceil+1\right)\left(\left \lfloor\frac{1}{2}{n-2\choose
k-1}\right\rfloor+1\right).
$$
\end{theorem}

\begin{proof}
$(1)$ Let $uv$ be an edge of $G$. Since $G$ is $(n-k)$-connected, it
follows that for any $S\subseteq V(G)$ and $|S'|=k-1$, $d_G(S\cup
\{u\})=d_G(S\cup \{v\})=k$, and hence $|N_u(e;k)|=|N_v(e;k)|=0$. So
$Sz_{k}(G)=m$.

$(2)$ From the definition, we have
$$
Sz_{k}(G) = \sum_{uv\in E(G)} (n_u(e;k)+1) (n_v(e;k)+1)\geq
\sum_{e=uv\in E(G)}1=e(G)=m.
$$
and
\begin{eqnarray*}
Sz_{k} (G)&=&\sum_{uv\in E(G)} (n_u(e;k)+1) (n_v(e;k)+1)\\[2mm]
&\leq &\sum_{uv\in E(G)}\left(\left\lceil\frac{1}{2}{n-2\choose
k-1}\right\rceil+1\right)\left(\left \lfloor\frac{1}{2}{n-2\choose
k-1}\right\rfloor+1\right)\\[2mm]
&=&m\left(\left\lceil\frac{1}{2}{n-2\choose
k-1}\right\rceil+1\right)\left(\left \lfloor\frac{1}{2}{n-2\choose
k-1}\right\rfloor+1\right).
\end{eqnarray*}
\end{proof}

\end{document}